\documentclass{article}

\usepackage{lineno,hyperref}
\usepackage{amssymb,amsmath, amsfonts, amsthm}

\usepackage{graphicx}
\usepackage{pgfplots}
\usepackage{tikz}
\usepackage{booktabs}
\usepackage{array}
\usepackage{verbatim}
\usepackage{subfig}
\usepackage{geometry}
\usepackage{fancyhdr}

\newcommand{\vertex}{\node[vertex]}
\tikzstyle{vertex}=[circle, draw, inner sep=0pt, minimum size=6pt]
\newtheorem{theorem}{Theorem}
\newtheorem{lemma}{Lemma}
\newtheorem{corollary}{Corollary}

\modulolinenumbers[5]

\begin{document}

\title{Power domination in cubic graphs and Cartesian products}
\author{
Sarah E. Anderson$^{a}$
\and
K. Kuenzel$^{b}$\\
}

\date{}

\maketitle

\begin{center}
$^a$ Department of Mathematics, University of St. Thomas, St. Paul, MN 55105\\
$^b$ Department of Mathematics, Trinity College, Hartford, CT 06106\\

\end{center}
\vskip15mm

\begin{abstract}
The power domination problem focuses on finding the optimal placement of phase measurement units (PMUs) to monitor an electrical power network. In the context of graphs,  the power domination number of a graph $G$, denoted $\gamma_P(G)$, is the minimum number of vertices needed to observe every vertex in the graph according to a specific set of observation rules. In \cite{ZKC_cubic}, Zhao et al. proved that if $G$ is a connected claw-free cubic graph of order $n$, then $\gamma_P(G) \leq n/4$. In this paper, we show that  if $G$ is a claw-free diamond-free cubic graph of order $n$, then $\gamma_P(G) \le n/6$, and this bound is sharp. We also provide new bounds on $\gamma_P(G \Box H)$ where $G\Box H$ is the Cartesian product of graphs $G$ and $H$. In the specific case that $G$ and $H$ are trees whose power domination number and domination number are equal, we show the Vizing-like inequality holds and $\gamma_P(G \Box H) \ge \gamma_P(G)\gamma_P(H)$.  
\end{abstract}

\section{Introduction} \label{sec:intro}
Power domination is a graph searching process that arose from the problem of monitoring an electrical power network. Phase measurement units (PMUs) are placed along the power network to observe the information using a two step process. Due to the cost, the goal is to use as few PMUs as possible while still being able to monitor the entire power network. This problem has been studied in various contexts, and it was first investigated using graphs by Haynes et al. in \cite{electricgrids}.

Haynes et al. modeled an electrical power network using a graph where the vertices represent the electric nodes and the edges represent transmission lines between two electrical nodes. Let $G = (V(G), E(G))$ be the simple graph representing this electrical power network, and let $S \subseteq V(G)$ be a subset of vertices that represents where the PMUs are placed. We label all vertices in $S$ as well as all the edges incident to a vertex in $S$ as observed. We then label other vertices and edges in the graph as observed using the following rules.

\begin{itemize}
\item Any vertex that is incident to an observed edge is observed.
\item Any edge joining two observed vertices is observed.
\item If a vertex is incident to a total of $k > 1$ edges and if $k - 1$ of these edges are observed, then all $k$ of these edges are observed.
\end{itemize}

If $S$ eventually observes all of the vertices and edges in the graph, then we say $S$ is a {\it power dominating set of $G$}. Since our goal is to use the smallest number of PMUs while observing the entire graph, we want to find the smallest set $S$ that power dominates the graph. The cardinality of the the smallest power dominating set of a graph $G$ is called the power domination number, which is denoted by $\gamma_P(G)$. 

Over the years, this original definition of power domination has been reshaped. In \cite{gridgraphs}, it was noted that one only needs to consider when all the vertices are observed. Since then, the definition of power domination has simplified so that only the vertices of the graph are observed. We present this new definition as well for ease of reference. Let $G = (V(G), E(G))$ be a simple graph. Due to the physical nature of this problem, we will assume $G$ has finite order in this paper. Let $S \subseteq V(G)$. We label the vertices in the graph as observed using the following rules.

\begin{itemize}
\item \textbf{Initialization Step (Domination Step):} All vertices in $S$ as well as all neighbors of vertices in $S$ are observed.
\item \textbf{Propagation Step (Zero Forcing Step):} Every vertex which is the only unobserved neighbor of some observed vertex becomes observed.
\end{itemize}

Note that while the initialization step can only occur once, the propagation step can occur as many times as needed. As before, $S$ is a {\it power dominating set} if all the vertices of the graph are eventually observed. Recall that a set $D\subseteq V(G)$ is referred to as a dominating set of $G$ if $V(G) = N[D]$. Therefore, the above initialization step is referred to as the {\it domination step} since we say $S$ dominates $\cup_{s\in S}N[s]$ in domination theory. Furthermore, the cardinality of the smallest dominating set of a graph $G$ is called the {\it domination number} and is denoted by $\gamma(G)$. If $D$ is a dominating set of cardinality $\gamma(G)$, we refer to $D$ as a $\gamma(G)$-set. The propagation step is called the {\it zero forcing step} since this process without the domination step is known as zero forcing. Zero forcing has been extensively studied, and it is notable for its connection to the minimum rank problem \cite{AIM}. A set that observes the entire graph using just the zero forcing process (or the propagation step) is known as a zero forcing set. The cardinality of the smallest zero forcing set of a graph $G$ is called the {\it zero forcing number}, which is denoted by $Z(G)$. Since any dominating set or zero forcing set of a graph $G$ is also a power dominating set of $G$, $\gamma_P(G) \leq \min \{\gamma(G), Z(G)\}$.

The power domination number of various graphs has been studied (see \cite{PDGP_BF, gridgraphs, PDGP_LCCK}). We will focus our attention on cubic graphs as well as the Cartesian product of graphs. 

First, we consider the power domination number of specific cubic graphs. In \cite{ZKC_cubic}, Zhao et al. proved that if $G$ is a connected claw-free cubic graph of order $n$, then $\gamma_P(G) \leq n/4$. Dorbec et al. extended this result in \cite{Dorbec_cubic} and showed the following.

\begin{theorem} \cite{Dorbec_cubic}  Let $G$ be a connected  cubic graph on $n$ vertices. If $G$ is not the complete bipartite
graph $K_{3,3}$, then $\gamma_P(G) \leq n/4$.
\end{theorem}

Recall that we define a diamond to be the graph $K_4 - e$ for any edge $e$ in $K_4$. We improve the upper bound given above as follows.

\begin{theorem} \label{thm:uppern3}
If $G$ is a claw-free diamond-free cubic graph of order $n$, then $\gamma_P(G) \le n/6$ and this bound is sharp. 
\end{theorem}

Next, we investigate the relationship between the power domination number, the domination number, and the zero forcing number in the  Cartesian product of two graphs. For graphs $G$ and $H$, the Cartesian product $G\Box H$ has vertex set
$V(G \Box H) = \{(g,h)\,:\, g\in V(G), h \in V(H)\}$.  Two vertices $(g_1,h_1)$ and $(g_2,h_2)$ are adjacent in $G\Box H$ if either $g_1=g_2$ and $h_1h_2\in E(H)$ or $h_1=h_2$ and $g_1g_2 \in E(G)$. We provide new lower and upper bounds on $\gamma_P(G \Box H)$ for two connected graphs $G$ and $H$. In \cite{pdprod}, Koh and Soh claim that  for any graph $G$ and any tree $T$ $\gamma_P(G\Box T) \ge \gamma_P(G)\gamma_P(T)$. However,  we provide an example that illustrates a flaw in their proof technique. It may however be the case that the statement is true, and in fact, we show the following.

\begin{theorem} \label{thm:treeprod}
If $T_1$ and $T_2$ are trees with $\gamma_P(T_1)= \gamma(T_1)$ and $\gamma_P(T_2) = \gamma(T_2)$, then $\gamma_P(T_1 \Box T_2) \ge \gamma_P(T_1)\gamma_P(T_2)$.  
\end{theorem}

The remainder of this paper is organized as follows. In Section \ref{cubicgraphs}, we provide new upper bounds on the power domination number of claw-free diamond-free cubic graphs. In Section \ref{cartesianproducts}, we consider the power domination number of the Cartesian product of two graphs. We conclude this paper with Section \ref{conc} and provide some future directions and open problems.

\section{Cubic graphs}
\label{cubicgraphs}
 In \cite{ZKC_cubic}, Zhao et al. provided an upper bound for the power domination of a claw-free cubic graph and classified all such graphs that achieve the upper bound as follows. We refer to a {\it diamond} as the graph $D$ obtained from the complete graph $K_4$ by deleting an edge. In \cite{ZKC_cubic}, they defined $D_k$ for each positive integer $k$ to be the connected claw-free cubic graph formed from $k$ disjoint copies of $D$ by joining pairwise $2k$ vertices of degree two. Furthermore, they let $\mathcal{A} = \{D_k \mid k \ge 1\}$. 
 
 \begin{theorem}\cite{ZKC_cubic}\label{thm:claw-free} If $G$ is a connected claw-free cubic graph of order $n$, then $\gamma_P(G) \le n/4$ with equality if and only if $G \in \mathcal{A}$.
 \end{theorem}

From the above, one can see that if $G$ is a connected claw-free cubic graph, the presence of diamonds in $G$ increases the power domination number of $G$. The goal of this section is to show that if $G$ is a connected claw-free diamond-free cubic graph of order $n$, then $\gamma_P(G) \le n/6$. To do so, we will use the following result. 

\begin{theorem}\cite{Oum} A graph $G$ is $2$-edge-connected claw-free cubic if and only if either
\begin{enumerate}
\item[(i)] $G \cong K_4$,
\item[(ii)] $G$ is a ring of diamonds, or
\item[(iii)] $G$ can be built from a $2$-edge-connected cubic multigraph $H$ by replacing some edges of $H$ with strings of diamonds and replacing each vertex of $H$ with a triangle.
\end{enumerate}
\end{theorem}

Note that since we focus on claw-free and diamond-free cubic graphs, we will use the following corollary. 
\begin{corollary}\cite{Oum}\label{cor:2-factor} A graph $G$ is $2$-edged-connected claw-free diamond-free cubic if and only if $G$ can be built from a $2$-edged-connected cubic multigraph $H$ by replacing each vertex with a triangle. 
\end{corollary}

When considering a $2$-edge-connected claw-free diamond-free cubic graph $G$, we will often talk about the corresponding cubic multigraph $H$ from which $G$ can be built by replacing each vertex of $H$ with a triangle. Note that Petersen \cite{Petersen} showed that $H$ indeed contains a $2$-factor in the following result. Furthermore, we can build a $2$-factor for $G$ by taking any $2$-factor from the corresponding cubic multigraph $H$ by replacing each vertex with a triangle. 

\begin{theorem}\cite{Petersen}\label{thm:Petersen} Every bridgeless cubic multigraph contains a $2$-factor. 
\end{theorem}

In addition to the above, we will need to use modified versions of some well-known results regarding $2$-edge-connected cubic graphs. In \cite{Plesnik}, Plesn\'{i}k points out that Sch\"{o}nberger  proved the following in 1934. 

\begin{theorem}\cite{Schonberger} Every bridgeless cubic graph has a $1$-factor not containing two arbitrarily chosen edges.
\end{theorem}

Plesn\'{i}k generalized the above result  as follows.

\begin{theorem}\cite{Plesnik}\label{thm:Plesnik} Let $G$ be an $(r-1)$-edge-connected regular graph of degree $r>0$ where $|V(G)|$ is even and let $H$ be an arbitrary set of $r-1$ edges. The graph $G' = G - H$ has a $1$-factor.
\end{theorem}

The above result holds in the specific case where $G$ is a $2$-edge-connected cubic multigraph. We restate a variation of the above result in the specific case where $G$ is a $2$-edge-connected cubic multigraph here as we will use it repeatedly throughout this section. We leave it to the reader to verify that the proof of Theorem~\ref{thm:Plesnik} can be extended as follows.

\begin{theorem}\cite{Plesnik}\label{thm:multigaph-1factor} Let $G$ be a $2$-edge-connected cubic multigraph where $|V(G)|$ is even. For any arbitrary edge $e \in E(G)$, $G$ contains a $2$-factor containing $e$. 
\end{theorem}

As one can surmise, we will need to consider power dominating sets in $2$-edge-connected cubic multigraphs. Suppose we start with a set $S \subset V(G)$ where $G$ is a multigraph and $v$ is an observed vertex with only one unobserved neighbor $w$. If there are multiple edges between $v$ and $w$, then technically, $w$ would not be observed by $v$ according to the original observation rules where we require all edges and vertices to be observed. Therefore, when considering multigraphs, special care must be used in the propagation step in that we cannot merely assume that because $w$ is the only unobserved neighbor of $v$ that $w$ will in fact be observed. Ultimately, this will require that we are able to carefully choose $2$-factors in a $2$-edge-connected cubic  multigraph to avoid issues in the propagation step. We are now ready to consider the power domination number of a claw-free diamond-free cubic graph. We first focus on the special case when $G$ is also $2$-edge-connected.

\begin{theorem}\label{thm:2connected} If $G$ is a $2$-edge-connected claw-free diamond-free cubic graph of order $n$, then $\gamma_P(G) \le \frac{n}{6}$ and this bound is sharp. 
\end{theorem}

\begin{proof} By Corollary~\ref{cor:2-factor}, we may assume $G$ can be obtained from a $2$-edge-connected cubic multigraph $H$ by replacing each vertex of $H$ with a triangle. By Theorem~\ref{thm:Petersen}, $H$ contains a $2$-factor $\mathcal{C'} = C_1' \cup \cdots \cup C_k'$. We let $\mathcal{C}= C_1 \cup \cdots \cup C_k$ be the $2$-factor  in $G$ obtained from $\mathcal{C'}$ where $C_i$ is obtained from $C_i'$ by replacing each vertex with a triangle. We may enumerate the vertices of $C_i$ as $C_i = x^i_1\dots x^i_{n_i}$ such that $3 \mid n_i$ and $G[\{x^i_j, x^i_{j+1}, x^i_{j-1}\}] = K_3$ for all $j \equiv 1\pmod{3}$. Note that by construction, $|V(C_i)| \ge 6$ for all $1\le i \le k$.

Moreover, there exists a perfect matching between the vertices of 
\[\bigcup_{i\in[k]} \{x^i_1, x^i_4, x^i_7, \dots, x^i_{n_i-2}\}.\]

Let $M = \{w_1v_1, \dots , w_tv_t\}$ be such a perfect matching.  We claim that there exists a way to choose one vertex from each edge in $M$, call the resulting set $D$, such that  $D\cap V(C_i) \ne \emptyset$ for all $i \in [k]$. Indeed, let $J$ be the graph where $V(J) = \{u_1, \dots, u_k\}$ and $u_iu_j\in E(J)$ if and only if there exists an edge $xy \in E(G)$ such that $x$ is on $C_i$ and $y$ is on $C_j$, $i\ne j$. Let $A_r = \{v \in V(J)\mid d_J(u_1, v) = r\}$. Note that each edge in $J$ corresponds to some edge in $M$. Let $M' = \{w_{\alpha_1}v_{\alpha_1}, \dots, w_{\alpha_s}v_{\alpha_s}\}$ be a subset of $M$ where $s = |E(J)|$ and $w_{\alpha_j}v_{\alpha_j} \in M'$ if and only if $w_{\alpha_j}$ is on $C_{\ell}$,  $v_{\alpha_j}$ is on $C_{\ell'}$ and $u_{\ell}u_{\ell'} \in E(J)$. Moreover, we can interchange $w_{\alpha_j}$ and $v_{\alpha_j}$ so that $d_J(u_1, u_{\ell})\le d_J(u_1, u_{\ell'})$. We shall assume $d = \max_v\{ {\rm{dist}}(u_1, v)\}$. Choose a path $P = y_1\dots y_d$ where $u_1 = y_1$ and $y_d \in A_d$. Thus, $y_j \in A_j$ for $j\in\{2, \dots, d\}$. Note that we may reindex the vertices of $J$ so that $u_i = y_i$ for $1 \le i \le d$. Furthermore, we may reindex $C_1, \dots, C_k$ so that $u_i$ in $J$ represents $C_i$ in $G$. Reordering if necessary, we shall assume $\{w_{\alpha_1}v_{\alpha_1}, \dots, w_{\alpha_{d-1}}v_{\alpha_{d-1}}\} \subseteq M'$ where $w_{\alpha_i}$ is on $C_i$, $v_{\alpha_i}$ is on $C_{i+1}$ for $1 \le i \le d-1$, and $w_{\alpha_i}v_{\alpha_i} \in M'$ because $y_iy_{i+1}$ is an edge in $P$.  Consider the set \[D = \{w_{\alpha_1}, \dots, w_{\alpha_{d-1}}, v_{\alpha_d}, \dots, v_{\alpha_s}\} \cup \{w_j\mid w_jv_j \not\in M'\}.\] If $D \cap V(C_i) \ne \emptyset$ for all $i \in[k]$, then we are done. Therefore, we assume there exists some $j' \in [k]$ such that $D\cap V(C_{j'})=\emptyset$. By construction, $D$ contains a vertex from each cycle $C_1, \dots, C_{d-1}$. Now consider $C_j$ for $j \in \{d, \dots, k\}$. This cycle corresponds to the vertex $u_j\in V(J)$. Assume ${\rm{dist}}(u_1, u_j) = \ell$. Therefore, we can find a path $P' = a_1\dots a_{\ell +1} $ where $a_1 = u_1$ and $a_{\ell+1} = u_j$. $a_{\ell}a_{\ell+1}$ corresponds to an edge $w_{\alpha_r}v_{\alpha_r}$ in $M'$ where $v_{\alpha_r}$ is on $C_j$ since $a_{\ell+1} = u_j$. Therefore, as long as $j \not\in [d]$, then $v_{\alpha_r} \in D$. Finally, we must show that $D$ contains a vertex from $C_{d}$. If $C_d$ contains two vertices in $\{x^d_1, x^d_4, \dots, x^d_{n_d-2}\}$ that are adjacent, then $D$ contains a vertex from $C_d$. So we shall assume that every vertex in $\{x^d_1, x^d_4, \dots, x^d_{n_d-2}\}$ is adjacent to some vertex not on $C_d$. There exists an edge in $J$ between $u_d$ and some vertex $u_s$ where $u_s$ is not on $P$. So there exists a corresponding edge $w_{\alpha_p}v_{\alpha_p} \in M'$. If $u_s \in A_{d-1}$, then $v_{\alpha_p} \in D$ and $v_{\alpha_p}$ is on $C_d$. So we shall assume that $u_s \in A_d$. If $v_{\alpha_p}$ is on $C_d$, then we are done. If not, then $w_{\alpha_p}$ is on $C_d$ and we simply select $D' = D - \{v_{\alpha_p}\} \cup \{w_{\alpha_p}\}$ and now $D'$ is such that $D' \cap V(C_i) \ne \emptyset$ for each $i\in [k]$. 

Let $D$ be any set of vertices that contains exactly one vertex from each edge in $M$ such that $D\cap V(C_i) \ne \emptyset$ for $i \in [k]$. We claim that $D$ is a power dominating set of $G$. Note that each vertex of the form $x_j^i$ for $i \in [k]$ and $j \equiv 1\pmod{3}$ is dominated by $D$. Fix $i \in [k]$ and reindex the vertices of $C_i$ if need be so that $x_1^i \in D$. Thus, $x_2^i$ and $x_{n_i}^i$ are dominated. It follows that $x_3^i$ is observed. Furthermore, $x_4^i$ is dominated by $D$ from which it follows that $x_5^i$ will be observed. Continuing this argument, we see that all vertices of $C_i$ will be observed and $D$ is in fact a power dominating set of $G$. Furthermore, we can partition the vertices of $G$ as $V(G) = X_1\cup \cdots \cup X_t$ where the following is true. For each $i \in [t]$, $w_i$ and $v_i$ are in $X_i$. Moreover, if we assume $w_i = x_j^a$ and $v_i = x_{\ell}^b$, then $x_{j+1}^a, x_{j-1}^a, x_{\ell -1}^b, x_{\ell+1}^b$ are also in $X_i$. Note that $|D \cap X_i|=1$ for all $i \in [t]$ and it follows that $|D| = \frac{n}{6}$. 

Finally, to see that the bound is sharp, consider the graph $G$ depicted in Figure~\ref{fig:sharp}. One can easily verify that $\gamma_P(G)>1$ and $\{u,v\}$ is a power dominating set of $G$. 
\end{proof}

\begin{figure}[h]
\begin{center}
\begin{tikzpicture}[]
\tikzstyle{vertex}=[circle, draw, inner sep=0pt, minimum size=6pt]
\tikzset{vertexStyle/.append style={rectangle}}
	\vertex (1) at (0,0) [scale=.75] {};
	\vertex (2) at (0,1) [ scale=.75] {};
	\vertex (3) at (1, 2) [ scale=.75, label=above:$u$] {};
	\vertex (4) at (2, 1) [scale=.75] {};
	\vertex (5) at (2, 0) [scale=.75] {};
	\vertex (6) at (1, -1) [scale=.75] {};
	\vertex (7) at (3, 0) [scale=.75] {};
	\vertex (8) at (3, 1) [scale=.75] {};
	\vertex (9) at (4, 2) [scale=.75] {};
	\vertex (10) at (5, 1) [scale=.75] {};
	\vertex (11) at (5,0) [scale=.75] {};
	\vertex (12) at (4, -1) [scale=.75, label=below:$v$] {};

	\path
		(1) edge (2)
		(2) edge (3)
		(3) edge (4)
		(4) edge (5)
		(5) edge (6)
		(1) edge (6)
		(2) edge (4)
		(1) edge (5)
		(3) edge (9)
		(7) edge (8)
		(8) edge (9)
		(9) edge (10)
		(10) edge (11)
		(11) edge (12)
		(7) edge (12)
		(6) edge (12)
		(7) edge (11)
		(8) edge (10)

	;
\end{tikzpicture}
\end{center}
\caption{An example where $\gamma_P(G) = \frac{|V(G)|}{6}$}
\label{fig:sharp}
\end{figure}
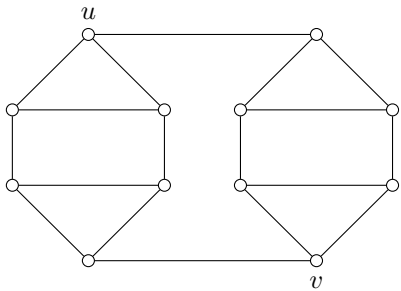

\begin{corollary}\label{cor:multigraph} If $G$ is a $2$-edge-connected claw-free diamond-free  cubic multigraph of order $n$ containing exactly one pair of vertices $\{u,v\}$ where there are two edges between $u$ and $v$, then $\gamma_P(G) \le (n-2)/6$.
\end{corollary}

\begin{proof}
Suppose $G$ contains vertices $u$ and $v$ where there are two edges between $u$ and $v$. Let $w$ be the neighbor of $u$ different from $v$ and let $z$ be the neighbor of $v$ different from $u$. Note that $w\ne z$ for otherwise $G$ is not $2$-edge-connected. Let $G'$ be the graph obtained from $G$ by removing $u$ and $v$ and adding an edge between $w$ and $z$. Since $G$ is claw-free, $w$ is on a triangle in $G$, say $wrs$. If $z = r$, then $s$ is a cut-vertex of $G$ which is a contradiction. Therefore, $z \not\in \{r, s\}$. On the other hand, $z$ is on a triangle in $G$, say $zab$. If $a=r$, then $s=b$ and $wrsz$ induces a diamond in $G$, another contradiction. Therefore, $\{a, b\} \cap \{r, s\} = \emptyset$. We may conclude that $G'$ is a $2$-edge-connected claw-free diamond-free cubic graph. By Corollary~\ref{cor:2-factor}, we may assume $G'$ can be obtained from a $2$-edge-connected cubic multigraph $H$ by replacing each vertex of $H$ with a triangle. Note that $wz$ is an edge in $H$ as it is not on a triangle in $G'$. Since $H$ is a cubic multigraph, $H$ has an even number of vertices. By Theorem~\ref{thm:multigaph-1factor}, $H$ contains a $2$-factor $\mathcal{C'} = C_1' \cup \cdots \cup C_k'$ containing $wz$. Reindexing if necessary, we may assume $wz$ is on $C_1'$. We let $\mathcal{C}= C_1 \cup \cdots \cup C_k$ be the $2$-factor in $G'$ obtained from $\mathcal{C}'$ where $C_i$ is obtained from $C_i'$ by replacing each vertex with a triangle. We may enumerate the vertices of $C_i$ as $C_i = x^i_1\dots x^i_{n_i}$ such that $3 \mid n_i$ and $G[\{x^i_j, x^i_{j+1}, x^i_{j-1}\}] = K_3$ for all $j \equiv 1\pmod{3}$. Note that there exists a perfect matching between the vertices of 
\[\bigcup_{i\in[k]} \{x^i_1, x^i_4, x^i_7, \dots, x^i_{n_i-2}\}.\] 
Let $M = \{w_1v_1, \dots , w_tv_t\}$ be such a perfect matching. As in the proof of Theorem~\ref{thm:2connected}, we can find a set $D$ of vertices  that contains exactly one vertex from each edge in $M$ such that $D\cap V(C_i) \ne \emptyset$ for $i \in [k]$.  We claim that $D$ is also a power dominating set of $G$. Note that since $wrs$ is a triangle in $G'$ and $zab$ is a triangle in $G'$ where $\{a, b\} \cap \{r, s\} = \emptyset$, we may assume $z = x_{n_1-1}^1$ and $w = x_{n_1}^1$. Therefore, $\mathcal{C''} = C_1'' \cup \cdots \cup C_k''$ where $C_i''= C_i$ for $2 \le i \le k$ and $C_1'' = x_1^1x_2^1\dots x_{n_1-1}^1vux_{n_1}^1$ is a $2$-factor of $G$ where $D \cap V(C_i'') \ne \emptyset$ for all $i \in [k]$. One can easily verify that all vertices of $G$ are observed. It is important to note that $D\cap \{u,v,w,z\} = \emptyset$ and $w$ will eventually observe $u$ and $z$ will eventually observe $v$ in $G$. Thus, $D$ is a power dominating set of $G$ of cardinality $|V(G')|/6 = (n-2)/6$.
\end{proof}

In order to generalize the above result to all cubic graphs, we use the following terminology found in \cite{zero-sum}.
Let $G$ be a 2-edge-connected graph with $2 \leq \delta(G) < \Delta(G) = 3$   and  vertex $v$ of degree 2. Then the graph that results by {\it smoothing $v$}, denoted $s_G(v)$, is the multigraph that is produced by  replacing $v$ and its two incident edges with an edge between the neighbors of $v$.   We note that $s_G(v)$ is a 2-edge-connected cubic multigraph.

Let $G$ be a  graph with non-empty bridge set $B(G)$. For each edge  $e$ in  $B(G)$, there exist distinct components $H_i$ and $H_j$ of $G - B(G)$ such that $e$ is incident to some vertex in $V(H_i)$ and some vertex in $V(H_j)$. In such a case, we will say that {\it $e$ is incident to  $H_i$ and $H_j$}. 

Let $G$ be a connected graph with $\vert B(G) \vert = b \geq 0$,  and let $H_0, H_1, ..., H_b$ be the components of $G - B(G)$.   Then  $T_G$ shall denote the simple graph with vertex set $V(T_G) = \{h_0, h_1, h_2, ..., h_b\}$ and edge set  $E(T_G)= \{h_ih_j \mid $ some edge $e \in B(G)$ is incident to $H_i$ and $H_j\}$. We observe that $T_G$ is a tree, and that  $G$ has no bridges if and only if  $T_G$ is isomorphic to $K_1$.
Furthermore, $G - B(G)$  is a graph with $b + 1$ components $H_0, H_1, H_2,  ..., H_b$ such that each component is either
\vskip 5pt
\indent Type I:  isomorphic to $K_1$, or
\vskip 5pt
\indent Type II:  isomorphic to the $m$-cycle $C_m$ for some $m \geq 2$, or
\vskip 5pt
\indent Type III: isomorphic to some 2-edge-connected graph $H$ with $\Delta(H) = 3$.
\vskip 5pt

\noindent We note that if $H_i$ is of Type I or Type II, then $b > 0$ and $h_i$ is an interior vertex of $T_G$.  Otherwise, if $H_i$ is of Type III, then  $h_i$ is a leaf of $T_G$ if and only if either $b = 0$ or $H_i$ has precisely one vertex of degree 2.   To illustrate, we observe that for graph $G_0$ of Figure 2, $T_{G_0}$ is isomorphic to $K_{1,3}$ and $G_0 - B(G_0)$ has one component of Type II and three components of Type III. 

\begin{figure}[h]
\begin{center}
\begin{tikzpicture}[]
\tikzstyle{vertex}=[circle, draw, inner sep=0pt, minimum size=6pt]
\tikzset{vertexStyle/.append style={rectangle}}
	\vertex (1) at (0.75,0.75) [fill, scale=.75] {};
	\vertex (2) at (2.25,0.75) [fill, scale=.75] {};
	\vertex (3) at (1.5,1.5) [fill, scale=.75]{};
	\vertex (4) at (1.5, 2.5) [fill, scale=.75]{};
	\vertex (5) at (2.25, 3.5) [fill, scale=.75]{};
	\vertex (6) at (.75, 3.5) [fill, scale=.75]{};
	\vertex (7) at (-.25, .75) [fill, scale=.75]{};
	\vertex (8) at (-1.25, 1.5) [fill, scale=.75]{};
	\vertex (9) at (-1.25, 0) [fill, scale=.75]{};
	\vertex (10) at (3.25, .75) [fill, scale=.75]{};
	\vertex (11) at (4.25, 1.5) [fill, scale=.75]{};
	\vertex (12) at (4.25, 0) [fill, scale=.75]{};

	\path
		(1) edge (2)
		(2) edge (3)
		(1) edge (3)
		(3) edge (4)
		(4) edge (5)
		(4) edge (6)
		(5) edge (6)
		(5) edge[bend right=35] (6)
		(1) edge (7)
		(7) edge (8)
		(7) edge (9)
		(8) edge (9)
		(8) edge[bend right=35] (9)
		(2) edge (10)
		(10) edge (11)
		(10) edge (12)
		(11) edge (12)
		(11) edge[bend left=35] (12)

	;
\end{tikzpicture}
\end{center}
\caption{ The graph $G_0$ }
\end{figure}
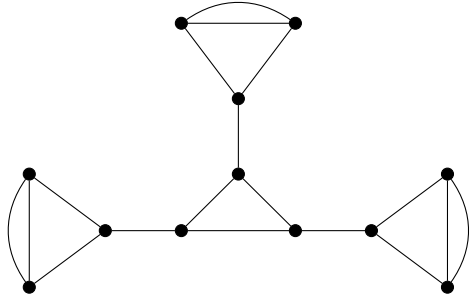

\medskip
\noindent \textbf{Theorem~\ref{thm:uppern3}} \emph{If $G$ is a claw-free diamond-free cubic graph of order $n$, then $\gamma_P(G) \le n/6$ and this bound is sharp.}

\begin{proof} We may assume that $G$ is connected since the power domination number of a disconnected graph is the sum of the power domination numbers of each component. We let $B(G)$ represent the set of bridges in $G$ and let $H_0, H_1, \dots, H_b$ be the components of $G - B(G)$. Note that since $G$ is claw-free, no $H_i$ is of Type I. Moreover, if $H_i$ is of Type II, then $H_i = K_3$. We consider the tree $T_G$ with $V(T_G) = \{h_0, h_1, \dots, h_b\}$ where $h_ih_j \in E(T_G)$ if and only if some $e \in B(G)$ is incident to $H_i$ and $H_j$. Choose a path $P$ of length $d={\rm{diam}}(T_G)$ and let $u_0$ be a leaf of $P$. We root $T_G$ at $u_0$ and define $A_i = \{v \in V(T_G) \mid d_{T_G}(u_0, v) = i\}$ where $A_0 = \{u_0\}$. Note that for each $a \in A_i$, there exists exactly one neighbor of $a$ in $A_{i-1}$. We also know that every vertex of $A_d$ is a leaf in $T_G$. We construct a power dominating set as follows. 

 Let $u_i\in V(T_G)$ be a leaf. Note that $H_i$ is of Type III and there exists only one vertex of degree $2$, call it $v$. It follows that  $s_{H_i}(v)$ is a claw-free diamond-free cubic multigraph with only one pair of vertices, namely $\{w, t\}$ where $N_{H_i}(v) = \{w, t\}$, that have two edges between them. Therefore, by Corollary~\ref{cor:multigraph}, we can choose a power dominating set $S_i'$ for $s_{H_i}(v)$ of cardinality at most $(n(H_i)-2)/6$. Furthermore, as noted at the end of the proof of Corollary~\ref{cor:multigraph}, $w$ will eventually be observed by its only neighbor in $s_{H_i}(v)$ other than $t$ and $t$ will eventually be observed by its only neighbor in $s_{H_i}(v)$ other than $w$. Thus, $w$ will observe $v$ in $H_i$ and $S_i'$ is a power dominating set of $H_i$.  Note that we have chosen a set of vertices from each $H_i$ where $u_i \in A_d$. 
 
 Next, suppose $u_i\in A_t$ for $1 \le t \le d-1$ and assume $H_i$ is of Type III. Let $X = \{x_1, \dots, x_r\}$ be all vertices of $H_i$ of degree $2$ such that   $x_1$ is the only vertex that has a neighbor in $H_j$ where $u_j \in A_{t-1}$. Note that $r \ge 2$. 
 
 \begin{enumerate}
 \item[Case 1] Assume first that $r$ is even. Let $\overline{H_i}$ be the graph obtained from $H_i$ by adding the edges $x_jx_{j+1}$ for $1 \le j \le r$ where $j$ is odd. Note that $\overline{H_i}$ is a cubic graph. We claim that $\overline{H_i}$ is claw-free and diamond-free. Suppose to the contrary that  $\overline{H_i}$ contains a claw. In particular, assume $\{u, v, w, x\}$ induces a claw where $u, v$, and $w$ are the leaves of the claw. Since $G$ is claw-free, it must be that some edge of this claw was added to form $\overline{H_i}$. Therefore, assume that $xu \not\in E(G)$. Thus, $x$ has a neighbor in $H_{\ell}$ for some $\ell \ne i$, call it $t$. But now $\{v, w, t, x\}$ induces a claw in $G$ which cannot be. Thus, $\overline{H_i}$ is claw-free.

 Now we show that $\overline{H_i}$ is diamond-free. Suppose to the contrary that $\{x, v, w, t\}$ induces a diamond in $\overline{H_i}$, call it $X$. We shall assume $xt \not\in E(X)$ and therefore $xt \not\in E(G)$. Since $G$ is diamond-free, some edge exists in $X$ that does not exist in $G$. Suppose first that $xv \not\in E(G)$.  This implies $x$ has a neighbor in $H_{\ell}$, call it $s$,  where $\ell \ne i$ and $x$ has a neighbor $z \not\in \{s, w\}$ such that $z\in H_i$ or $z \in H_{\ell'}$ where $\ell'\not\in \{\ell, i\}$ . This implies $N_G[x]$ induces a claw in $G$ which is a contradiction. Therefore, we shall assume $wv \not\in E(G)$. Thus, $w$ has a neighbor in some $H_{\ell}$ where $\ell\ne i$ and it follows that $N_G[w]$ induces a claw, another contradiction. Hence, $\overline{H_i}$ is diamond-free. 
 
 As in the proof of Theorem~\ref{thm:2connected}, we can find a $2$-factor $\mathcal{C}=C_1 \cup \cdots \cup C_k$ of $\overline{H_i}$ where $|V(C_j)|\ge6$ for $j\in [k]$. We may enumerate the vertices of $C_j$ as $C_j = x_1^j\dots x_{n_j}^j$ such that $3 \mid n_j$ and $G[\{x_s^j, x_{s+1}^j, x_{s-1}^j\}]=K_3$ for all $s\equiv 1\pmod{3}$. Moreover, we can find a power dominating set $S_i'$ of cardinality $|\overline{H_i}|/6$ such that $D\cap V(C_j) \ne \emptyset$ for each $j \in [k]$.

 \item[Case 2] Assume $r$ is odd. Let $H_i'$ be the graph obtained from $H_i$ by adding the edges $x_jx_{j+1}$ for $2 \le j \le r$ where $j$ is even and let $\overline{H_i} = s_{H_i'}(x_1)$. Thus, $\overline{H_i}$ is a $2$-edge-connected cubic multigraph with exactly one pair of vertices that have two edges between them. In fact, if we let $x$ and $y$ be the neighbors of $x_1$ in $H_i$, then $xy \in E(G)$ for otherwise $N_G[x_1]$ is a claw. It follows that $x$ and $y$ have exactly two edges between them in $\overline{H_i}$. 
 
 As in Case 1, $\overline{H_i}$ is claw-free.  Next, we show that $\overline{H_i}$ is in fact diamond-free. Suppose to the contrary that $\{u, v, w, t\}$ induces a diamond in $\overline{H_i}$, call it $X$. We may assume $ut \not\in E(\overline{H_i})$ and therefore $ut \not\in E(G)$. Note that $\{u,v,w,t\} \cap \{x, y\} = \emptyset$ since $x$ and $y$ share two edges between them. Now similar arguments to that used in Case 1 show that $X$ cannot exist unless $G$ contains a claw. Thus,  $\overline{H_i}$ is diamond-free. 
 
As in the proof of Corollary~\ref{cor:multigraph}, we can find a $2$-factor $\mathcal{C}=C_1 \cup \cdots \cup C_k$ of $\overline{H_i}$ where $|V(C_j)|\ge6$ for $j\in [k]$. We may enumerate the vertices of $C_j$ as $C_j = x_1^j\dots x_{n_j}^j$ such that $3 \mid n_j$ for $2 \le j \le k$, $n_1\equiv 2\pmod{3}$, $G[\{x_s^j, x_{s+1}^j, x_{s-1}^j\}]=K_3$ for all $s\equiv 1\pmod{3}$ where $s\ne x_{n_1-1}^1$, and $C_1 = x_1^1x_2^1\dots x_{n_1-3}^1xyx_{n_1}^1$. Moreover, we can find a power dominating set $S_i'$ of $\overline{H_i}$ of cardinality $(|\overline{H_i}|-2)/6 $ such that $D \cap V(C_j) \ne \emptyset$ for each $j \in [k]$.

 \end{enumerate}
 
 All we need to do now is prove that $S' = \cup_i S_i'$ is a power dominating set of $G$. Let $u_i \in A_t$ for some $0 \le t \le d$. Note that if $t \in \{0, d\}$, we have already shown that all vertices of $H_i$ are observed. Assume first that $t=d-1$. If $H_i$ is of Type II, then $H_i = K_3$. We write $V(H_i) = \{x_1, x_2, x_3\}$ where $x_1$ has a neighbor $w$ in $H_{\ell}$ where $u_{\ell} \in A_{d-2}$. It follows that $x_2$ has a neighbor $y$ in $H_{\ell'}$ where $u_{\ell'} \in A_d$. Since $y$ is observed, $x_2$ is the only unobserved neighbor of $y$. Thus, $x_2$ is observed. Similarly, $x_3$ is observed. It follows that $x_1$ is also observed. Therefore, we may assume that $H_i$ is of Type III. Let $X = \{x_1, \dots, x_r\}$ be all the vertices of $H_i$ of degree $2$ where $x_1$ has a neighbor in $H_{\ell}$ where $u_{\ell}\in A_{d-2}$. 
 
 We first consider when $r$ is even. Define $\overline{H_i}$ as above in Case 1 with $2$-factor $\mathcal{C}= C_1 \cup \cdots \cup C_k$ of $\overline{H_i}$ where $|V(C_j)| \ge 6$ for $j \in [k]$ and $S' \cap V(C_j) \ne \emptyset$. Moreover, enumerate the vertices of $C_j$ as $C_j = x_1^j\dots x_{n_j}^j$ such that $3 \mid n_j$ and $G[\{x_{\ell}^j, x_{\ell+1}^j, x_{\ell-1}^j\}]=K_3$ for all $\ell \equiv 1\pmod{3}$. Note that $\mathcal{C}$ is still a $2$-factor in $H_i$. Furthermore, each $x_{\alpha}$ where $2 \le \alpha \le r$ is observed by some vertex in $H_p$ where $u_p \in A_d$. Since $S' \cap V(C_j) \ne \emptyset$ for each $j \in [k]$, it follows that all vertices of $H_i$ except possibly $x_1$ are observed. We may assume $x_1 = x_{a}^j$ for some $a\equiv 1\pmod{3}$ and $1 \le j \le k$. Since all vertices in $N_{H_i}(x_{\alpha}^j)$ are observed, $x_a^j$ will also be observed.  Thus, all vertices of $H_i$ are observed. 
 
 Finally, suppose $r$ is odd. Define $\overline{H_i}$ as above in Case 2 with $2$-factor $\mathcal{C}=C_1 \cup \cdots \cup C_k$ of $\overline{H_i}$ where $|V(C_j)|\ge6$ for $j\in [k]$ and $S' \cap V(C_j) \ne \emptyset$.  For $2\le j \le k$, enumerate the vertices of $C_j$ as $C_j = x_1^j\dots x_{n_j}^j$ such that $3 \mid n_j$ and $G[\{x_s^j, x_{s+1}^j, x_{s-1}^j\}]=K_3$ for all $s\equiv 1\pmod{3}$. Moreover, write  $C_1 = x_1^1x_2^1\dots x_{n_1-3}^1xyx_{n_1}^1$ such that $x$ and $y$ share two edges between them and $G[\{x_s^1, x_{s+1}^1, x_{s-1}^1\}]=K_3$ for all $s \equiv 1\pmod{3}$. Note that the $2$-factor $\mathcal{C'} = C_1' \cup \cdots \cup C_k'$ defined as $C_j' = C_j$ for $2 \le j \le k$ and $C_1' = x_1^1x_2^1\dots x_{n_1-3}^1xx_1yx_{n_1}^1$ is a $2$-factor of $H_i$ where $S' \cap V(C_j') \ne \emptyset$ for each $j \in [k]$. Each $x_{\alpha}$ where $2 \le \alpha \le r$ is observed by some vertex in $H_p$ where $u_p \in A_d$. Moreover, $S' \cap V(C_j)\ne \emptyset$ for each $j \in [k]$, $x$ will eventually be observed by its only neighbor in $\overline{H_i}$ other than $y$, and $y$ will eventually be observed by its only neighbor in $\overline{H_i}$ other than $x$. Thus, $x$ will observe $x_1$ in $H_i$ and  all vertices of $H_i$ will be observed.

 Therefore, we have shown that for any $u_i \in A_{d-1}$, all vertices of $H_i$ are observed. Repeating these same arguments will show that all vertices of $G$ will eventually be observed as we work our way up the tree $T_G$. It follows that $S'$ is a power dominating set of $G$ where $|S'| \le \frac{n}{6}$. The sharpness of the bound is provided in Theorem~\ref{thm:2connected}.

\end{proof}

\section{Cartesian products}
\label{cartesianproducts}
We now turn our attention to the Cartesian products of two connected graphs $G$ and $H$. Note that the results presented in this section may be extended to disconnected graphs since the power domination of a disconnected graph is the sum of the power domination numbers of each component. First, we will consider lower bounds for $\gamma_P(G\Box H)$. Note that the following was shown by Koh and Soh in \cite{pdprod}. 

\begin{lemma}\cite{pdprod}\label{thm:Cartlower} For any connected graphs $G$ and $H$, $\gamma_P(G\Box H) \ge \gamma_P(H)$.
\end{lemma}

In \cite{pdprod}, it was also claimed that $\gamma_P(G\Box T) \ge \gamma_P(G)\gamma_P(T)$ where $T$ is any tree. The proof provided relies on the following. A {\it spider} is a tree with at most one vertex having degree $3$ or more. The {\it spider number} of a tree $T$, denoted ${\rm sp}(T)$, is the minimum number of subsets into which $V(T)$ can be partitioned so that each subset induces a spider. Such a partition is referred to as a {\it spider partition} and each set of the partition a {\it spider subset}. The following was proven by Haynes et al. \cite{electricgrids}.

\begin{theorem}\cite{electricgrids}\label{thm:spider} For any tree $T$, $\gamma_P(T) = {\rm sp}(T)$.
\end{theorem}

Koh and Soh \cite{pdprod} used the above result to show $\gamma_P(G\Box T) \ge \gamma_P(G)\gamma_P(T)$. However, there is a hole in their proof, which we provide here in order to address the issue. 

\begin{theorem}\cite{pdprod}\label{thm:Cartwrong} For any graph $G$ and any tree $T$, $\gamma_P(G)\gamma_P(T) \le \gamma_P(G\Box T)$. 
\end{theorem}

\begin{proof} Let $S$ be a power dominating set of $G \Box T$ where $m = \gamma_P(G)$ and $n = \gamma_P(T)$. By Theorem~\ref{thm:spider}, we can partition $V(T)$ into $n$ spider subsets $V_1, V_2, \dots, V_n$. Let $T_i = G[V_i]$, the subgraph induced by $V_i$, where $1 \le i \le n$. By Lemma~\ref{thm:Cartlower}, $\gamma_P(G\Box T_i) \ge \gamma_P(G) = m$. It follows that $|S \cap V(G\Box T_i)| \ge m$, and therefore 
\[\gamma_P(G\Box T) = |S| = \sum_{i=1}^n |S \cap V(G \Box T_i)| \ge mn.\]
\end{proof}

The issue in the above proof is the conclusion that $\gamma_P(G\Box T_i) \ge \gamma_P(G) = m$ implies that $|S \cap V(G\Box T_i)| \ge m$. We provide an example where this need not be the case. Consider the graph $G\Box T$ depicted in Figure~\ref{fig:Cart} where $G$ is the graph of order $7$. We claim that ${\rm sp}(T) = 4$. To see this, suppose there exists a spider partition of size $3$, say $V_1, V_2,$ and $V_3$. By definition, $V_1$, $V_2$, and $V_3$ must each induce a tree. Thus, we may assume that $v_3 \in V_1$, $v_{12} \in V_2$ and $v_{18} \in V_3$. Note that if $v_9$ is in $V_1$, then $V_1$ will contain two vertices of degree $3$, $v_3$ and $v_5$, since $V_1$ must be connected. So $v_9 \notin V_1$. A similar argument can be made to show $v_9 \notin V_3$. Therefore, $v_9 \in V_2$. In addition, $v_8 \notin V_2$; otherwise, $V_2$ contains two vertices, $v_{12}$ and $v_9$, of degree $3$. If $v_8 \in V_1$, then $v_7 \in V_1$. So $V_1$ now contains two vertices, $v_3$ and $v_5$, each of degree $3$. Hence, it must the case of $v_8 \in V_3$. However, $V_3$ now contains two vertices, namely $v_{16}$ and $v_{18}$, each with degree $3$. Therefore, no such spider partition exists. However, $V_1 = \{v_1, v_2, v_3, v_4, v_5, v_6\}$, $V_2 = \{v_7, v_8, v_{15}\}$, $V_3 = \{v_9, v_{10}, v_{11}, v_{12}, v_{13}, v_{14}\}$, and $V_4 = \{v_{16}, v_{17}, v_{18}, v_{19}, v_{20}\}$ is a spider partition. Furthermore,  $S = V(G) \times \{v_3, v_5, v_9, v_{12}, v_{16}, v_{18}\}$, illustrated by the black vertices, is a power dominating as it is a dominating set of $G\Box T$ yet $S \cap (V(G) \times V_2) = \emptyset$. 

It is quite possible that the statement of Theorem~\ref{thm:Cartwrong} is correct. In fact, we can show the Vizing-like inequality holds in specific situations. Recall that given a tree $T$, we refer to any vertex adjacent to a leaf as a {\it support vertex} and any vertex adjacent to at least two leaves as a {\it strong support vertex}. Note that strong support vertices can be used to bound the power domination number of a graph. For example, assume $x$ is a strong support vertex of $G$ with leaves $\ell_1$ and $\ell_2$ and $D$ is power dominating set of $G$. Let $X = \{x, \ell_1, \ell_2\}$ and assume $D \cap X = \emptyset$. Then the only vertex of $X$ adjacent to a vertex outside of $X$ is $x$. Even if $x$ is observed at some time step, $\ell_1$ and $\ell_2$ can never be observed by the rules of the propagation step which is a contradiction. In fact, if we let $v_s(G)$ denote the number of strong support vertices in a graph $G$, then $v_s(G) \leq \gamma_P(G)$ using similar arguments. We use the following result shown in \cite{electricgrids} and strong support vertices to show the Vizing-like inequality holds in specific situations.

\begin{figure}[h!]
\begin{center}
\begin{tikzpicture}[scale=.75]
\vertex (1) at (1, -1) [scale=.75pt, label=below:$u_1$]{};
\vertex (2) at (3, -1) [scale=.75pt, label=below:$u_2$]{};
\vertex (3) at (5, -1) [scale=.75pt, label=below:$u_3$]{};
\vertex (4) at (7, -1) [scale=.75pt, label=below:$u_4$]{};
\vertex (5) at (9, -1) [,scale=.75pt, label=below:$u_5$]{};
\vertex (6) at (11, -1) [scale=.75pt, label=below:$u_6$]{};
\vertex (7) at (13, -1) [scale=.75pt, label=below:$u_7$]{};

\vertex (a) at (-1,1) [scale=.75pt, label=left:$v_1$]{};
\vertex (b) at (-1,2) [scale=.75pt, label=left:$v_2$]{};
\vertex (c) at (-1,3) [scale=.75pt, label=left:$v_3$]{};
\vertex (d) at (-1,4) [scale=.75pt, label=left:$v_4$]{};
\vertex (e) at (-1,5) [scale=.75pt, label=left:$v_5$]{};
\vertex (f) at (-1,6) [scale=.75pt, label=left:$v_6$]{};
\vertex (g) at (-1,7) [scale=.75pt, label=left:$v_7$]{};
\vertex (h) at (-1,8) [scale=.75pt, label=left:$v_8$]{};
\vertex (i) at (-1,9) [scale=.75pt, label=left:$v_9$]{};
\vertex (j) at (-1,10) [scale=.75pt, label=left:$v_{10}$]{};
\vertex (k) at (-1,11) [scale=.75pt, label=left:$v_{11}$]{};
\vertex (l) at (-1,12) [scale=.75pt, label=left:$v_{12}$]{};
\vertex (m) at (-1,13) [scale=.75pt, label=left:$v_{13}$]{};
\vertex (n) at (-1,14) [scale=.75pt, label=left:$v_{14}$]{};
\vertex (o) at (-1,15) [scale=.75pt, label=left:$v_{15}$]{};
\vertex (p) at (-1,16) [scale=.75pt, label=left:$v_{16}$]{};
\vertex (q) at (-1,17) [scale=.75pt, label=left:$v_{17}$]{};
\vertex (r) at (-1,18) [scale=.75pt, label=left:$v_{18}$]{};
\vertex (s) at (-1,19) [scale=.75pt, label=left:$v_{19}$]{};
\vertex (t) at (-1,20) [scale=.75pt, label=left:$v_{20}$]{};

\vertex (1a) at (1, 1) [scale=.75pt, label=below:$$]{};
\vertex (2a) at (3, 1) [scale=.75pt, label=below:$$]{};
\vertex (3a) at (5, 1) [scale=.75pt, label=below:$$]{};
\vertex (4a) at (7, 1) [scale=.75pt, label=below:$$]{};
\vertex (5a) at (9, 1) [scale=.75pt, label=below:$$]{};
\vertex (6a) at (11, 1) [scale=.75pt, label=below:$$]{};
\vertex (7a) at (13, 1) [scale=.75pt, label=below:$$]{};

\vertex (1b) at (1, 2) [scale=.75pt, label=below:$$]{};
\vertex (2b) at (3, 2) [scale=.75pt, label=below:$$]{};
\vertex (3b) at (5, 2) [scale=.75pt, label=below:$$]{};
\vertex (4b) at (7, 2) [scale=.75pt, label=below:$$]{};
\vertex (5b) at (9, 2) [scale=.75pt, label=below:$$]{};
\vertex (6b) at (11, 2) [scale=.75pt, label=below:$$]{};
\vertex (7b) at (13, 2) [scale=.75pt, label=below:$$]{};

\vertex (1c) at (1, 3) [scale=.75pt, fill=black, label=below:$$]{};
\vertex (2c) at (3, 3) [scale=.75pt, fill=black,label=below:$$]{};
\vertex (3c) at (5, 3) [scale=.75pt, fill=black,label=below:$$]{};
\vertex (4c) at (7, 3) [scale=.75pt, fill=black,label=below:$$]{};
\vertex (5c) at (9,3) [scale=.75pt, fill=black,label=below:$$]{};
\vertex (6c) at (11, 3) [scale=.75pt, fill=black,label=below:$$]{};
\vertex (7c) at (13, 3) [scale=.75pt,fill=black, label=below:$$]{};

\vertex (1d) at (1, 4) [scale=.75pt, label=below:$$]{};
\vertex (2d) at (3, 4) [scale=.75pt, label=below:$$]{};
\vertex (3d) at (5, 4) [scale=.75pt, label=below:$$]{};
\vertex (4d) at (7, 4) [scale=.75pt, label=below:$$]{};
\vertex (5d) at (9, 4) [scale=.75pt, label=below:$$]{};
\vertex (6d) at (11, 4) [scale=.75pt, label=below:$$]{};
\vertex (7d) at (13, 4) [scale=.75pt, label=below:$$]{};

\vertex (1e) at (1, 5) [scale=.75pt, fill=black, label=below:$$]{};
\vertex (2e) at (3, 5) [scale=.75pt,  fill=black,label=below:$$]{};
\vertex (3e) at (5, 5) [scale=.75pt,  fill=black,label=below:$$]{};
\vertex (4e) at (7, 5) [scale=.75pt,  fill=black,label=below:$$]{};
\vertex (5e) at (9, 5) [scale=.75pt,  fill=black,label=below:$$]{};
\vertex (6e) at (11, 5) [scale=.75pt,  fill=black,label=below:$$]{};
\vertex (7e) at (13, 5) [scale=.75pt,  fill=black,label=below:$$]{};

\vertex (1f) at (1, 6) [scale=.75pt, label=below:$$]{};
\vertex (2f) at (3, 6) [scale=.75pt, label=below:$$]{};
\vertex (3f) at (5, 6) [scale=.75pt, label=below:$$]{};
\vertex (4f) at (7, 6) [scale=.75pt, label=below:$$]{};
\vertex (5f) at (9, 6) [scale=.75pt, label=below:$$]{};
\vertex (6f) at (11, 6) [scale=.75pt, label=below:$$]{};
\vertex (7f) at (13, 6) [scale=.75pt, label=below:$$]{};

\vertex (1g) at (1, 7) [scale=.75pt, label=below:$$]{};
\vertex (2g) at (3, 7) [scale=.75pt, label=below:$$]{};
\vertex (3g) at (5, 7) [scale=.75pt, label=below:$$]{};
\vertex (4g) at (7, 7) [scale=.75pt, label=below:$$]{};
\vertex (5g) at (9, 7) [scale=.75pt, label=below:$$]{};
\vertex (6g) at (11, 7) [scale=.75pt, label=below:$$]{};
\vertex (7g) at (13, 7) [scale=.75pt, label=below:$$]{};

\vertex (1h) at (1, 8) [scale=.75pt, label=below:$$]{};
\vertex (2h) at (3, 8) [scale=.75pt, label=below:$$]{};
\vertex (3h) at (5, 8) [scale=.75pt, label=below:$$]{};
\vertex (4h) at (7, 8) [scale=.75pt, label=below:$$]{};
\vertex (5h) at (9, 8) [scale=.75pt, label=below:$$]{};
\vertex (6h) at (11, 8) [scale=.75pt, label=below:$$]{};
\vertex (7h) at (13, 8) [scale=.75pt, label=below:$$]{};

\vertex (1i) at (1, 9) [scale=.75pt,  fill=black,label=below:$$]{};
\vertex (2i) at (3, 9) [scale=.75pt,fill=black, label=below:$$]{};
\vertex (3i) at (5, 9) [scale=.75pt,fill=black, label=below:$$]{};
\vertex (4i) at (7, 9) [scale=.75pt,fill=black, label=below:$$]{};
\vertex (5i) at (9, 9) [scale=.75pt, fill=black,label=below:$$]{};
\vertex (6i) at (11, 9) [scale=.75pt,fill=black, label=below:$$]{};
\vertex (7i) at (13, 9) [scale=.75pt, fill=black,label=below:$$]{};

\vertex (1j) at (1, 10) [scale=.75pt, label=below:$$]{};
\vertex (2j) at (3, 10) [scale=.75pt, label=below:$$]{};
\vertex (3j) at (5, 10) [scale=.75pt, label=below:$$]{};
\vertex (4j) at (7, 10) [scale=.75pt, label=below:$$]{};
\vertex (5j) at (9, 10) [scale=.75pt, label=below:$$]{};
\vertex (6j) at (11, 10) [scale=.75pt, label=below:$$]{};
\vertex (7j) at (13, 10) [scale=.75pt, label=below:$$]{};

\vertex (1k) at (1, 11) [scale=.75pt, label=below:$$]{};
\vertex (2k) at (3, 11) [scale=.75pt, label=below:$$]{};
\vertex (3k) at (5, 11) [scale=.75pt, label=below:$$]{};
\vertex (4k) at (7, 11) [scale=.75pt, label=below:$$]{};
\vertex (5k) at (9, 11) [scale=.75pt, label=below:$$]{};
\vertex (6k) at (11, 11) [scale=.75pt, label=below:$$]{};
\vertex (7k) at (13, 11) [scale=.75pt, label=below:$$]{};

\vertex (1l) at (1, 12) [scale=.75pt, fill=black, label=below:$$]{};
\vertex (2l) at (3, 12) [scale=.75pt, fill=black, label=below:$$]{};
\vertex (3l) at (5, 12) [scale=.75pt, fill=black, label=below:$$]{};
\vertex (4l) at (7, 12) [scale=.75pt, fill=black, label=below:$$]{};
\vertex (5l) at (9, 12) [scale=.75pt, fill=black, label=below:$$]{};
\vertex (6l) at (11, 12) [scale=.75pt, fill=black, label=below:$$]{};
\vertex (7l) at (13, 12) [scale=.75pt, fill=black, label=below:$$]{};

\vertex (1m) at (1, 13) [scale=.75pt, label=below:$$]{};
\vertex (2m) at (3, 13) [scale=.75pt, label=below:$$]{};
\vertex (3m) at (5, 13) [scale=.75pt, label=below:$$]{};
\vertex (4m) at (7, 13) [scale=.75pt, label=below:$$]{};
\vertex (5m) at (9, 13) [scale=.75pt, label=below:$$]{};
\vertex (6m) at (11, 13) [scale=.75pt, label=below:$$]{};
\vertex (7m) at (13, 13) [scale=.75pt, label=below:$$]{};

\vertex (1n) at (1, 14) [scale=.75pt, label=below:$$]{};
\vertex (2n) at (3, 14) [scale=.75pt, label=below:$$]{};
\vertex (3n) at (5, 14) [scale=.75pt, label=below:$$]{};
\vertex (4n) at (7, 14) [scale=.75pt, label=below:$$]{};
\vertex (5n) at (9, 14) [scale=.75pt, label=below:$$]{};
\vertex (6n) at (11, 14) [scale=.75pt, label=below:$$]{};
\vertex (7n) at (13, 14) [scale=.75pt, label=below:$$]{};

\vertex (1o) at (1, 15) [scale=.75pt, label=below:$$]{};
\vertex (2o) at (3, 15) [scale=.75pt, label=below:$$]{};
\vertex (3o) at (5, 15) [scale=.75pt, label=below:$$]{};
\vertex (4o) at (7, 15) [scale=.75pt, label=below:$$]{};
\vertex (5o) at (9, 15) [scale=.75pt, label=below:$$]{};
\vertex (6o) at (11, 15) [scale=.75pt, label=below:$$]{};
\vertex (7o) at (13, 15) [scale=.75pt, label=below:$$]{};

\vertex (1p) at (1, 16) [scale=.75pt, fill=black, label=below:$$]{};
\vertex (2p) at (3, 16) [scale=.75pt, fill=black, label=below:$$]{};
\vertex (3p) at (5, 16) [scale=.75pt, fill=black, label=below:$$]{};
\vertex (4p) at (7, 16) [scale=.75pt, fill=black, label=below:$$]{};
\vertex (5p) at (9, 16) [scale=.75pt, fill=black, label=below:$$]{};
\vertex (6p) at (11, 16) [scale=.75pt, fill=black, label=below:$$]{};
\vertex (7p) at (13, 16) [scale=.75pt, fill=black, label=below:$$]{};

\vertex (1q) at (1, 17) [scale=.75pt, label=below:$$]{};
\vertex (2q) at (3, 17) [scale=.75pt, label=below:$$]{};
\vertex (3q) at (5, 17) [scale=.75pt, label=below:$$]{};
\vertex (4q) at (7, 17) [scale=.75pt, label=below:$$]{};
\vertex (5q) at (9, 17) [scale=.75pt, label=below:$$]{};
\vertex (6q) at (11, 17) [scale=.75pt, label=below:$$]{};
\vertex (7q) at (13, 17) [scale=.75pt, label=below:$$]{};

\vertex (1r) at (1, 18) [scale=.75pt, fill=black, label=below:$$]{};
\vertex (2r) at (3, 18) [scale=.75pt, fill=black, label=below:$$]{};
\vertex (3r) at (5, 18) [scale=.75pt, fill=black, label=below:$$]{};
\vertex (4r) at (7, 18) [scale=.75pt, fill=black, label=below:$$]{};
\vertex (5r) at (9, 18) [scale=.75pt, fill=black, label=below:$$]{};
\vertex (6r) at (11, 18) [scale=.75pt, fill=black, label=below:$$]{};
\vertex (7r) at (13, 18) [scale=.75pt,  fill=black,label=below:$$]{};

\vertex (1s) at (1, 19) [scale=.75pt, label=below:$$]{};
\vertex (2s) at (3, 19) [scale=.75pt, label=below:$$]{};
\vertex (3s) at (5, 19) [scale=.75pt, label=below:$$]{};
\vertex (4s) at (7, 19) [scale=.75pt, label=below:$$]{};
\vertex (5s) at (9, 19) [scale=.75pt, label=below:$$]{};
\vertex (6s) at (11, 19) [scale=.75pt, label=below:$$]{};
\vertex (7s) at (13, 19) [scale=.75pt, label=below:$$]{};

\vertex (1t) at (1, 20) [scale=.75pt, label=below:$$]{};
\vertex (2t) at (3, 20) [scale=.75pt, label=below:$$]{};
\vertex (3t) at (5, 20) [scale=.75pt, label=below:$$]{};
\vertex (4t) at (7, 20) [scale=.75pt, label=below:$$]{};
\vertex (5t) at (9, 20) [scale=.75pt, label=below:$$]{};
\vertex (6t) at (11, 20) [scale=.75pt, label=below:$$]{};
\vertex (7t) at (13, 20) [scale=.75pt, label=below:$$]{};

\path

	(1) edge[bend left=30] (3) 
	(2) edge (3)
	(3) edge (4)
	(4) edge (5)
	(5) edge (6)
	(5) edge[bend left=30] (7)
	(1a) edge[bend left=30] (3a) 
	(2a) edge (3a)
	(3a) edge (4a)
	(4a) edge (5a)
	(5a) edge (6a)
	(5a) edge[bend left=30] (7a)
	(1b) edge[bend left=30] (3b) 
	(2b) edge (3b)
	(3b) edge (4b)
	(4b) edge (5b)
	(5b) edge (6b)
	(5b) edge[bend left=30] (7b)
	(1c) edge[bend left=30] (3c) 
	(2c) edge (3c)
	(3c) edge (4c)
	(4c) edge (5c)
	(5c) edge (6c)
	(5c) edge[bend left=30] (7c)
	(1d) edge[bend left=30] (3d) 
	(2d) edge (3d)
	(3d) edge (4d)
	(4d) edge (5d)
	(5d) edge (6d)
	(5d) edge[bend left=30] (7d)
	(1e) edge[bend left=30] (3e) 
	(2e) edge (3e)
	(3e) edge (4e)
	(4e) edge (5e)
	(5e) edge (6e)
	(5e) edge[bend left=30] (7e)
	(1f) edge[bend left=30] (3f) 
	(2f) edge (3f)
	(3f) edge (4f)
	(4f) edge (5f)
	(5f) edge (6f)
	(5f) edge[bend left=30] (7f)
	(1g) edge[bend left=30] (3g) 
	(2g) edge (3g)
	(3g) edge (4g)
	(4g) edge (5g)
	(5g) edge (6g)
	(5g) edge[bend left=30] (7g)
	(1h) edge[bend left=30] (3h) 
	(2h) edge (3h)
	(3h) edge (4h)
	(4h) edge (5h)
	(5h) edge (6h)
	(5h) edge[bend left=30] (7h)
	(1i) edge[bend left=30] (3i) 
	(2i) edge (3i)
	(3i) edge (4i)
	(4i) edge (5i)
	(5i) edge (6i)
	(5i) edge[bend left=30] (7i)
	(1j) edge[bend left=30] (3j) 
	(2j) edge (3j)
	(3j) edge (4j)
	(4j) edge (5j)
	(5j) edge (6j)
	(5j) edge[bend left=30] (7j)
	(1k) edge[bend left=30] (3k) 
	(2k) edge (3k)
	(3k) edge (4k)
	(4k) edge (5k)
	(5k) edge (6k)
	(5k) edge[bend left=30] (7k)
	(1l) edge[bend left=30] (3l) 
	(2l) edge (3l)
	(3l) edge (4l)
	(4l) edge (5l)
	(5l) edge (6l)
	(5l) edge[bend left=30] (7l)
	(1m) edge[bend left=30] (3m) 
	(2m) edge (3m)
	(3m) edge (4m)
	(4m) edge (5m)
	(5m) edge (6m)
	(5m) edge[bend left=30] (7m)
	(1n) edge[bend left=30] (3n) 
	(2n) edge (3n)
	(3n) edge (4n)
	(4n) edge (5n)
	(5n) edge (6n)
	(5n) edge[bend left=30] (7n)
	(1o) edge[bend left=30] (3o) 
	(2o) edge (3o)
	(3o) edge (4o)
	(4o) edge (5o)
	(5o) edge (6o)
	(5o) edge[bend left=30] (7o)
	(1p) edge[bend left=30] (3p) 
	(2p) edge (3p)
	(3p) edge (4p)
	(4p) edge (5p)
	(5p) edge (6p)
	(5p) edge[bend left=30] (7p)
	(1q) edge[bend left=30] (3q) 
	(2q) edge (3q)
	(3q) edge (4q)
	(4q) edge (5q)
	(5q) edge (6q)
	(5q) edge[bend left=30] (7q)
	(1r) edge[bend left=30] (3r) 
	(2r) edge (3r)
	(3r) edge (4r)
	(4r) edge (5r)
	(5r) edge (6r)
	(5r) edge[bend left=30] (7r)
	(1s) edge[bend left=30] (3s) 
	(2s) edge (3s)
	(3s) edge (4s)
	(4s) edge (5s)
	(5s) edge (6s)
	(5s) edge[bend left=30] (7s)
	(1t) edge[bend left=30] (3t) 
	(2t) edge (3t)
	(3t) edge (4t)
	(4t) edge (5t)
	(5t) edge (6t)
	(5t) edge[bend left=30] (7t)
	
	(a) edge[bend right=30] (c)
	(b) edge (c)
	(c) edge (d)
	(d) edge (e)
	(e) edge (f)
	(e) edge[bend right=30] (g)
	(g) edge (h)
	(h) edge (i)
	(i) edge (j)
	(i) edge[bend right=30] (k)
	(k) edge (l)
	(l) edge (m)
	(l) edge[bend right=30] (n)
	(h) edge[bend right=30] (o)
	(o) edge (p)
	(p) edge (q)
	(p) edge[bend right=30] (r)
	(r) edge (s)
	(r) edge[bend right=30] (t)
	
	(1a) edge[bend right=30] (1c)
	(1b) edge (1c)
	(1c) edge (1d)
	(1d) edge (1e)
	(1e) edge (1f)
	(1e) edge[bend right=30] (1g)
	(1g) edge (1h)
	(1h) edge (1i)
	(1i) edge (1j)
	(1i) edge[bend right=30] (1k)
	(1k) edge (1l)
	(1l) edge (1m)
	(1l) edge[bend right=30] (1n)
	(1h) edge[bend right=30] (1o)
	(1o) edge (1p)
	(1p) edge (1q)
	(1p) edge[bend right=30] (1r)
	(1r) edge (1s)
	(1r) edge[bend right=30] (1t)
	
	(2a) edge[bend right=30] (2c)
	(2b) edge (2c)
	(2c) edge (2d)
	(2d) edge (2e)
	(2e) edge (2f)
	(2e) edge[bend right=30] (2g)
	(2g) edge (2h)
	(2h) edge (2i)
	(2i) edge (2j)
	(2i) edge[bend right=30] (2k)
	(2k) edge (2l)
	(2l) edge (2m)
	(2l) edge[bend right=30] (2n)
	(2h) edge[bend right=30] (2o)
	(2o) edge (2p)
	(2p) edge (2q)
	(2p) edge[bend right=30] (2r)
	(2r) edge (2s)
	(2r) edge[bend right=30] (2t)
	
	(3a) edge[bend right=30] (3c)
	(3b) edge (3c)
	(3c) edge (3d)
	(3d) edge (3e)
	(3e) edge (3f)
	(3e) edge[bend right=30] (3g)
	(3g) edge (3h)
	(3h) edge (3i)
	(3i) edge (3j)
	(3i) edge[bend right=30] (3k)
	(3k) edge (3l)
	(3l) edge (3m)
	(3l) edge[bend right=30] (3n)
	(3h) edge[bend right=30] (3o)
	(3o) edge (3p)
	(3p) edge (3q)
	(3p) edge[bend right=30] (3r)
	(3r) edge (3s)
	(3r) edge[bend right=30] (3t)
	
	(4a) edge[bend right=30] (4c)
	(4b) edge (4c)
	(4c) edge (4d)
	(4d) edge (4e)
	(4e) edge (4f)
	(4e) edge[bend right=30] (4g)
	(4g) edge (4h)
	(4h) edge (4i)
	(4i) edge (4j)
	(4i) edge[bend right=30] (4k)
	(4k) edge (4l)
	(4l) edge (4m)
	(4l) edge[bend right=30] (4n)
	(4h) edge[bend right=30] (4o)
	(4o) edge (4p)
	(4p) edge (4q)
	(4p) edge[bend right=30] (4r)
	(4r) edge (4s)
	(4r) edge[bend right=30] (4t)
	
	(5a) edge[bend right=30] (5c)
	(5b) edge (5c)
	(5c) edge (5d)
	(5d) edge (5e)
	(5e) edge (5f)
	(5e) edge[bend right=30] (5g)
	(5g) edge (5h)
	(5h) edge (5i)
	(5i) edge (5j)
	(5i) edge[bend right=30] (5k)
	(5k) edge (5l)
	(5l) edge (5m)
	(5l) edge[bend right=30] (5n)
	(5h) edge[bend right=30] (5o)
	(5o) edge (5p)
	(5p) edge (5q)
	(5p) edge[bend right=30] (5r)
	(5r) edge (5s)
	(5r) edge[bend right=30] (5t)
	
	(6a) edge[bend right=30] (6c)
	(6b) edge (6c)
	(6c) edge (6d)
	(6d) edge (6e)
	(6e) edge (6f)
	(6e) edge[bend right=30] (6g)
	(6g) edge (6h)
	(6h) edge (6i)
	(6i) edge (6j)
	(6i) edge[bend right=30] (6k)
	(6k) edge (6l)
	(6l) edge (6m)
	(6l) edge[bend right=30] (6n)
	(6h) edge[bend right=30] (6o)
	(6o) edge (6p)
	(6p) edge (6q)
	(6p) edge[bend right=30] (6r)
	(6r) edge (6s)
	(6r) edge[bend right=30] (6t)
	
	(7a) edge[bend right=30] (7c)
	(7b) edge (7c)
	(7c) edge (7d)
	(7d) edge (7e)
	(7e) edge (7f)
	(7e) edge[bend right=30] (7g)
	(7g) edge (7h)
	(7h) edge (7i)
	(7i) edge (7j)
	(7i) edge[bend right=30] (7k)
	(7k) edge (7l)
	(7l) edge (7m)
	(7l) edge[bend right=30] (7n)
	(7h) edge[bend right=30] (7o)
	(7o) edge (7p)
	(7p) edge (7q)
	(7p) edge[bend right=30] (7r)
	(7r) edge (7s)
	(7r) edge[bend right=30] (7t)

				;	

\end{tikzpicture}
\end{center}
\caption{$G\Box T$}
\label{fig:Cart}
\end{figure}

\begin{theorem}[\cite{electricgrids}]\label{thm:treeequality} For any tree $T$ of order at least $3$, $\gamma_P(T) = \gamma(T)$ if and only if $T$ has a unique $\gamma(T)$-set $S$ and every vertex in $S$ is a strong support vertex.
\end{theorem}

\medskip
\noindent \textbf{Theorem \ref{thm:treeprod}}  \emph{If $T_1$ and $T_2$ are trees with $\gamma_P(T_1)= \gamma(T_1)$ and $\gamma_P(T_2) = \gamma(T_2)$, then $\gamma_P(T_1 \Box T_2) \ge \gamma_P(T_1)\gamma_P(T_2)$.}

\begin{proof} Let $S_1$ be a $\gamma(T_1)$-set  and let $S_2$ be a $\gamma(T_2)$-set. By Theorem~\ref{thm:treeequality}, every vertex in $S_1$ is a strong support vertex in $T_1$ and every vertex in $S_2$ is a strong support vertex in $T_2$. Enumerate the vertices of $S_1= \{x_1, \dots, x_k\}$ and enumerate the vertices of $S_2 = \{y_1, \dots, y_{\ell}\}$. Choose a partition $\Pi = \Pi_1\cup \cdots \cup \Pi_k$ for $V(T_1)$ such that $x_i \in \Pi_i$ and $\Pi_i \subseteq N_{T_1}[x_i]$ for each $1 \le i \le k$. Similarly, choose a partition $\Omega = \Omega_1 \cup \cdots \cup \Omega_{\ell}$ such that $y_i \in \Omega_i$ and $\Omega_i \subseteq N_{T_2}[y_i]$ for all $1 \le i \le \ell$. Define 
\[\mathcal{B}_{i,j} = \{(u,v) \mid u \in \Pi_i, v \in \Omega_j\}\]
for all $1 \le i \le k$, $1 \le j \le \ell$. Note that $\cup_{i,j}\mathcal{B}_{i,j}$ is a partition of $V(T_1 \Box T_2)$. Suppose there exists a power dominating set $D$ of $T_1 \Box T_2$ of cardinality $|S_1||S_2| -1$. This implies that $D \cap \mathcal{B}_{i,j} = \emptyset$ for some $i\in [k], j\in [\ell]$. Note that since $x_i$ is a strong support vertex of $T_1$, there exist two leaves, say $\ell_1$ and $\ell_2$, of $x_i$. Moreover, $\{\ell_1, \ell_2\} \subset \Pi_i$. Similarly, there exist two leaves, say $\ell'_1$ and $\ell'_2$, of $y_j$ and $\{\ell'_1, \ell'_2\} \subset \Omega_j$. 
 However, the only vertices of 
\[X_{ij} = \{(x_i, y_j), (x_i, \ell_1'), (x_i, \ell'_2), (\ell_1, y_j), (\ell_1, \ell'_1), (\ell_1, \ell'_2), (\ell_2, y_j), (\ell_2, \ell'_1), (\ell_2, \ell'_2)\}\]
adjacent to a vertex outside of $X$ are $(x_i, y_j), (x_i, \ell_1'), (x_i, \ell'_2), (\ell_1, y_j),$ and $(\ell_2, y_j)$. Even if all of these vertices are observed at some time step, $(\ell_1, \ell_1')$ will never be observed as both $(x_i, \ell_1')$ and $(\ell_1, y_j)$ each have two unobserved neighbors. Therefore, no such power dominating set of $T_1 \Box T_2$ exists and $\gamma_P(T_1\Box T_2) \ge \gamma_P(T_1)\gamma_P(T_2)$. 
\end{proof}

In Theorem  \ref{thm:treeprod}, we could provide a new lower bound on the power domination number of two trees $T_1$ and $T_2$ when  $\gamma_P(T_1)= \gamma(T_1)$ and $\gamma_P(T_2) = \gamma(T_2)$. By Theorem \ref{thm:treeequality}, we also know that the minimum power dominating sets of $T_1$ and $T_2$ consisted of strong support vertices. Since the number of strong support vertices is a lower bound on the power domination number of a graph, we know $\gamma_P(T_1) =  v_s(T_1)$ and  $\gamma_P(T_2) = v_s(T_2)$ in this case. We can also use strong support vertices to extend this result and bound the power domination number of general Cartesian products.

\begin{lemma}
For any connected graphs $G$ and $H$, $v_s(G)v_s(H) \leq \gamma_P(G \Box H)$. In addition, if $\gamma_P(G) = v_s(G)$ and $\gamma_P(H) = v_s(H)$, then $\gamma_P(G)\gamma_P(H) \leq \gamma_P(G \Box H)$.
\label{cartlowervs}
\end{lemma}
\begin{proof}
Let $D$ be a power dominating set of $G \Box H$. Assume $x_i$ is a strong support vertex of $G$ with leaves $\ell_1$ and $\ell_2$, and assume $y_i$ is strong support vertex of $H$ with leaves $\ell'_1$ and $\ell'_2$. Let 
\[
X_{ij} = \{(x_i, y_j), (x_i, \ell_1'), (x_i, \ell'_2), (\ell_1, y_j), (\ell_1, \ell'_1), (\ell_1, \ell'_2), (\ell_2, y_j), (\ell_2, \ell'_1), (\ell_2, \ell'_2)\}.
\]
First, assume $X_{ij} \cap D = \emptyset$. Note that the only vertices of $X_{ij}$ adjacent to a vertex outside of $X_{ij}$ are $(x_i, y_j), (x_i, \ell_1'), (x_i, \ell'_2), (\ell_1, y_j),$ and $(\ell_2, y_j)$. Even if all of these vertices are observed at some time step, it follows that $(\ell_1, \ell_1')$ will never be observed. Thus, it must be the case that $X_{ij}  \cap D \neq \emptyset$. Since similar arguments hold for any support vertices of $G$ and $H$, $v_s(G)v_s(H) \leq |D|$.
\end{proof}

For any graphs $G$ and $H$,  $\gamma_P(G \Box H) \leq Z(G \Box H )$. Therefore, the above result provides a lower bound on the zero forcing number of $G\Box H$. 

\begin{corollary}
For any connected graphs $G$ and $H$, $v_s(G)v_s(H) \leq Z(G \Box H)$. 
\end{corollary}

We will now consider the power domination number of certain Cartesian products. The hypercube graph, $Q_n$, can be defined recursively as $Q_n := Q_{n - 1} \Box K_2$ for $n \geq 1$, where $Q_0 = K_1$. In \cite{deanetal}, Dean et al. conjectured that $\gamma_P(Q_n) = \gamma(Q_{n - 1})$. While this conjecture was disproven in \cite{paichui}, an interesting question that arises is if there exists a graph $G$ such that $\gamma_P(G \Box K_2) = \gamma(G)$. In order to answer this question, we first turn our attention to upper bounds on $\gamma_P(G \Box H)$. A general upper bound of $\gamma_P(G \Box H)$ was provided in \cite{upperboundpowdom}.

\begin{theorem}\cite{upperboundpowdom}
For any connected graphs $G$ and $H$,
$$\gamma_P(G \Box H) \leq \min\{\gamma_P(G)|V(H)|, \gamma_P(H)|V(G)|\}.$$
\label{upperboundpowdom}
\end{theorem}

We construct a new upper bound based on the domination and zero forcing numbers of the graphs.

\begin{lemma}\label{lem:upperzero}
For any connected graphs $G$ and $H$,
\[\gamma_P(G \Box H) \leq \min\{\gamma(G)Z(H), \gamma(H)Z(G)\}.\]
\label{upperboundgammazero}
\end{lemma}
\begin{proof}
We will first show $\gamma_P(G \Box H) \leq \gamma(G)Z(H)$. Let $D$ be a minimum dominating set of $G$ and $Z$ be a minimum zero forcing set of $H$. We claim that $S := D \times H$ is a power dominating set of $G\Box H$. Note that each vertex in $V(G) \times Z$ is dominated by $S$. It follows that for each $g \in V(G)$, each vertex in $\{g\} \times V(H)$ will be observed as $Z$ is a zero forcing set of $H$.  Thus, $S$ is a power dominating set of $G \Box H$ and $\gamma_P(G \Box H) \leq \gamma(G)Z(H)$. Similarly, one can prove $\gamma_P(G \Box H) \leq \gamma(H)Z(G)$.
\end{proof}

When $Z(H) = 1$, Lemma~\ref{lem:upperzero} yields $\gamma_P(G \Box H) \leq \gamma(G)$. Note that $Z(H) = 1$ if and only if $H$ is a path. Similarly, when $\gamma(H) = 1$, Lemma~\ref{lem:upperzero} yields $\gamma_P(G \Box H) \leq Z(G)$. By Lemmas \ref{thm:Cartlower} and \ref{cartlowervs}, we also know that $\gamma_P(G) \leq \gamma_P(G \Box H)$ and $v_s(G)v_s(H) \le \gamma_P(G \Box H)$. This in turn allows us to find the power domination number of $\gamma_P(G \Box H)$ for certain classes of graphs.

\begin{corollary}
Let $G$ be a connected graph.
\begin{enumerate} 
\item If $\gamma_P(G) = \gamma(G)$, then $\gamma_P(G \Box P_n) = \gamma(G)$.
\item If $\gamma_P(G) = Z(G)$ and $\gamma(H) = 1$, then $\gamma_P(G \Box H) = Z(G)$.
\item If $v_s(G)= \gamma(G)$ and $v_s(H) = Z(H)$, then $\gamma_P(G \Box H) = \gamma(G)Z(H)$.
\label{cartprodpowdom}
\end{enumerate}
\end{corollary}

We can now turn our attention back to the question posed earlier: does there exist a graph $G$ such that $\gamma_P(G \Box K_2) = \gamma(G)$? From Corollary \ref{cartprodpowdom}, we know $\gamma_P(G \Box K_2) = \gamma(G)$ if $\gamma_P(G) = \gamma(G)$. For example, when $G = K_{m, n}$, then $\gamma_P(K_{m, n} \Box K_2) = \gamma(K_{m, n})$. In addition, if $\gamma_P(G) = Z(G)$, then $\gamma_P(G \Box K_2) = Z(G)$. Finally, note that Corollary \ref{cartprodpowdom} provides another example of a cubic graph where the bound given in Theorem \ref{thm:uppern3} is sharp. Consider $G = C_3 \Box K_2$. Since $\gamma_P(C_3) = \gamma(C_3)$, it follows that $\gamma_P(C_3 \Box K_2) = \gamma(C_3) =  \frac{n}{6}$.


\section{Conclusions}
\label{conc}
In this paper, we improved on the upper bound of the power domination number of  claw-free diamond-free cubic graphs. In particular, in Theorem \ref{thm:uppern3}, we showed that if $G$ is a claw-free diamond-free cubic graph of order $n$, then $\gamma_P(G) \le n/6$ and this bound is sharp. To do so, we created power dominating sets based on a specific $2$-factor of the graph where each cycle in the $2$-factor has length $6$ or more. One can verify that given such a $2$-factor $\mathcal{C} = C_1 \cup \cdots \cup C_k$ where $|V(C_i)| \ge 6$ for $i \in [k]$, that $\gamma_P(G) \ge k$. Thus, if there exists a $2$-factor of $G$ that consists only of cycles of length $6$, then $\gamma_P(G) = n/6$. An example of such a graph is shown in Figure~\ref{fig:sharp}. One open problem to consider is if these are the only such claw-free diamond-free cubic graphs whose power domination number equals this bound.

We also considered the power domination number of Cartesian products. We showed that a Vizing-like inequality holds in the case of certain trees. In fact, in Theorem \ref{thm:treeprod}, we proved that if $T_1$ and $T_2$ are trees with $\gamma_P(T_1)= \gamma(T_1)$ and $\gamma_P(T_2) = \gamma(T_2)$, then $\gamma_P(T_1 \Box T_2) \ge \gamma_P(T_1)\gamma_P(T_2)$. The next immediate question that comes to mind is if $\gamma(T_1\Box T_2) \ge \gamma_P(T_1)\gamma_P(T_2)$ for any trees $T_1$ and $T_2$? Additionally, we found graphs $G$ where $\gamma_P(G\Box K_2) = \gamma(G)$. Can we classify all graphs $G$ where $\gamma_P(G\Box K_2) = \gamma(G)$?

\end{document}